\documentclass[reqno,11pt,a4paper]{amsart}

\usepackage{amssymb,latexsym,enumerate,dsfont,mathrsfs,graphicx,url}
\usepackage[a4paper,margin=3cm]{geometry}

\title{Circular Law Theorem for Random Markov Matrices}

\date{Preprint -- May 2010 -- Compiled \today}


\author{Charles Bordenave}
\address[Ch.~Bordenave $\mathscr{C}_b$]{Universit\'e de Toulouse\\
  UMR 5219 CNRS Institut de Math\'ematiques de Toulouse\\
  118 route de Narbonne, F-31062 Toulouse, France}
\email{charles.bordenave(at)math.univ-toulouse.fr}
\urladdr{http://www.math.univ-toulouse.fr/~bordenave/}

\author{Pietro Caputo}
\address[P.~Caputo $\mathcal{P}(\mathds{C})$]{Dipartimento di Matematica\\
  Universit\`a Roma Tre\\
  Largo San Murialdo 1, 00146 Roma, Italy}
\email{caputo(at)mat.uniroma3.it}
\urladdr{http://www.mat.uniroma3.it/users/caputo/}

\author{Djalil Chafa\"\i}
\address[D.~Chafa{\"{\i}} $\mathcal{D}'(\mathds{C})$]{Universit\'e Paris-Est Marne-la-Vall\'ee\\
  UMR 8050 CNRS, 5 boulevard Descartes\\
  F-77454 Champs-sur-Marne, France}
\email{djalil(at)chafai.net} 
\urladdr{http://djalil.chafai.net/}

\newcommand{\ABS}[1]{{{\left| #1 \right|}}} 
\newcommand{\BRA}[1]{{{\left\{#1\right\}}}} 
\newcommand{\NRM}[1]{{{\left\| #1\right\|}}} 
\newcommand{\PAR}[1]{{{\left(#1\right)}}} 
\newcommand{\TV}[1]{\NRM{#1}_\textsc{TV}} 
\newcommand{\CEIL}[1]{{{\lceil#1\rceil}}} %
\newcommand{\FLOOR}[1]{{{\lfloor#1\rfloor}}} %
\renewcommand{\leq}{\leqslant}
\renewcommand{\geq}{\geqslant}
\newtheorem{thm}{Theorem}[section]%
\newtheorem{lem}[thm]{Lemma}%

\newtheorem{rem}[thm]{Remark}%
\numberwithin{equation}{section}

\begin{document}

\begin{abstract}
  Let $(X_{jk})_{jk\geq1}$ be i.i.d.\ nonnegative random variables with
  bounded density, mean $m$, and finite positive variance $\sigma^2$. Let $M$
  be the $n\times n$ random Markov matrix with i.i.d.\ rows defined by
  $M_{jk}=X_{jk}/(X_{j1}+\cdots+X_{jn})$. In particular, when $X_{11}$ follows
  an exponential law, the random matrix $M$ belongs to the Dirichlet Markov
  Ensemble of random stochastic matrices. Let $\lambda_1,\ldots,\lambda_n$ be
  the eigenvalues of $\sqrt{n}M$ i.e.\ the roots in $\mathds{C}$ of its
  characteristic polynomial. Our main result states that with probability one,
  the counting probability measure
  $\frac{1}{n}\delta_{\lambda_1}+\cdots+\frac{1}{n}\delta_{\lambda_n}$
  converges weakly as $n\to\infty$ to the uniform law on the disk
  $\{z\in\mathds{C}:|z|\leq m^{-1}\sigma\}$. The bounded density assumption is
  purely technical and comes from the way we control the operator norm of the
  resolvent.
\end{abstract}

\maketitle

{\small\tableofcontents}

\section{Introduction}

The \emph{eigenvalues} of an $n\times n$ complex matrix $A$ are the roots in
$\mathds{C}$ of its characteristic polynomial. We label them
$\lambda_1(A),\ldots,\lambda_n(A)$ so that
$|\lambda_1(A)|\geq\cdots\geq|\lambda_n(A)|$ with growing phases. The
\emph{spectral radius} is $|\lambda_1(A)|$. We also denote by
$s_1(A)\geq\cdots\geq s_n(A)$ the \emph{singular values} of $A$, defined for
all $1\leq k\leq n$ by $s_k(A):=\lambda_k(\sqrt{AA^*})$ where
$A^*=\bar{A}^\top$ is the conjugate-transpose. The matrix $A$ maps the unit
sphere to an ellipsoid, the half-lengths of its principal axes being the
singular values of $A$. The \emph{operator norm} of $A$ is
$$
\Vert A\Vert_{2\to2}:=\max_{\Vert x\Vert_2=1}\Vert Ax\Vert_2=s_1(A)
\quad\text{while}\quad
s_n(A)=\min_{\NRM{x}_2=1}\NRM{Ax}_2.
$$
The matrix $A$ is singular iff $s_n(A)=0$, and if not then
$s_n(A)=s_1(A^{-1})^{-1}=\Vert A^{-1}\Vert_{2\to 2}^{-1}$. If $A$ is normal
(i.e.\ $A^*A=A^*A$) then $s_i(A)=|\lambda_i(A)|$ for every $1\leq i\leq n$.
Beyond normal matrices, the relationships between the eigenvalues and the
singular values are captured by the Weyl inequalities (see lemma
\ref{le:weyl}). Let us define the discrete probability measures
$$
\mu_A:=\frac{1}{n}\sum_{k=1}^n\delta_{\lambda_k(A)}
\quad\text{and}\quad
\nu_A:=\frac{1}{n}\sum_{k=1}^n\delta_{s_k(A)}.
$$
From now on, we denote ``$\overset{\mathscr{C}_b}{\longrightarrow}$'' the weak
convergence of probability measures with respect to bounded continuous
functions.
We use the abbreviations \emph{a.s.}, \emph{a.a.}, and \emph{a.e.} for
\emph{almost surely}, \emph{Lebesgue almost all}, and \emph{Lebesgue almost
  everywhere} respectively. The notation $n\gg1$ means \emph{large enough
  $n$}. Let $(X_{i,j})_{i,j\geq1}$ be an infinite table of i.i.d.\ complex
random variables with finite positive variance $0<\sigma^2<\infty$. If one
defines the square $n\times n$ complex random matrix $X:=(X_{i,j})_{1\leq
  i,j\leq n}$ then the quartercircular law theorem (universal square version
of the Marchenko-Pastur theorem, see
\cite{marchenko-pastur,MR0467894,MR862241}) states that a.s.
\begin{equation}\label{eq:QCLT}
  \nu_{\frac{1}{\sqrt{n}}X}%
  \underset{n\to\infty}{\overset{\mathscr{C}_b}{\longrightarrow}}%
  \mathcal{Q}_\sigma
\end{equation} 
where $\mathcal{Q}_\sigma$ is the quartercircular law on the real interval
$[0,2\sigma]$ with Lebesgue density
$$
x\mapsto \frac{1}{\pi\sigma^2}\sqrt{4\sigma^2-x^2}\mathds{1}_{[0,2\sigma]}(x).
$$
Additionally, it is shown in \cite{MR1235416,MR963829,bai-silverstein-book}
that 
\begin{equation}\label{eq:s1X}
\lim_{n\to\infty} s_1({\scriptstyle\frac{1}{\sqrt{n}}}X)=2\sigma
\text{\ a.s.\ } %
\quad\text{iff}\quad %
\mathds{E}(X_{1,1})=0\quad\text{and}\quad\mathds{E}(|X_{1,1}|^4)<\infty.
\end{equation}
Concerning the eigenvalues, the famous Girko circular law theorem states that
a.s.\
\begin{equation}\label{eq:CLT}
\mu_{\frac{1}{\sqrt{n}}X}%
\underset{n\to\infty}{\overset{\mathscr{C}_b}{\longrightarrow}}%
\mathcal{U}_\sigma
\end{equation} 
where $\mathcal{U}_\sigma$ is the uniform law on the disc
$\{z\in\mathds{C}:|z|\leq\sigma\}$, known as the circular law. This statement
was established through a long sequence of partial results
\cite{MR0220494,MR773436,MR2130247,MR841088,MR1437734,MR1454343,MR1428519,bai-silverstein-book,1687963,gotze-tikhomirov-new,MR2409368,tao-vu-cirlaw-bis},
the general case \eqref{eq:CLT} being finally obtained by Tao and Vu
\cite{tao-vu-cirlaw-bis}. From \eqref{eq:CLT} we have a.s.\
$\varliminf_{n\to\infty}|\lambda_k(n^{-1/2}X)|\geq\sigma$ for any fixed
$k\geq1$, and we get from \cite{MR863545,1687963} and \eqref{eq:s1X} that if
additionally $\mathds{E}(X_{1,1})=0$ and $\mathds{E}(|X_{1,1}|^4)<\infty$ then
a.s.\
\begin{equation}\label{eq:silara}
  \lim_{n\to\infty}|\lambda_1({\scriptstyle\frac{1}{\sqrt{n}}}X)|=\sigma%
  \quad\text{and}\quad %
  \lim_{n\to\infty}\frac{s_1(X)}{|\lambda_1(X)|}=2.
\end{equation}
The behavior of the ratio operator-norm/spectral-radius suggests that $X$ is
far from being an asymptotically normal matrix. Following
\cite{MR1284550,djalil-nccl}, if $\mathds{E}(X_{1,1})\neq0$ while
$\mathds{E}(|X_{1,1}|^4)<\infty$ then a.s.\ $|\lambda_1(n^{-1/2}X)|\to+\infty$
at speed $\sqrt{n}$ while $|\lambda_2(n^{-1/2}X)|$ remains bounded.

The proof of \eqref{eq:CLT} is partly but crucially based on a polynomial
lower bound on the smallest singular value proved in \cite{MR2409368}: for
every $a,d>0$, there exists $b>0$ such that for any deterministic complex
$n\times n$ matrix $A$ with $s_1(A)\leq n^d$ we have
\begin{equation}\label{eq:sn-exp-pol}
  \mathds{P}(s_n(X+A)\leq n^{-b})\leq n^{-a}.
\end{equation}
In particular, by the first Borel-Cantelli lemma, there exists $b>0$ which may
depend on $d$ such that a.s.\ $X+A$ is invertible with $s_n(X+A)\geq n^{-b}$
for $n\gg1$.

\subsection{Random Markov matrices and main results}

From now on and unless otherwise stated $(X_{i,j})_{i,j\geq1}$ is an infinite
array of nonnegative real random variables with mean
$m:=\mathds{E}(X_{1,1})>0$ and finite positive variance
$\sigma^2:=\mathds{E}(X_{1,1}^2)-m^2$. Let us define the event
$$
\mathcal{D}_n:=\{\rho_{n,1}\cdots\rho_{n,n}>0\}
\quad\text{where}\quad
\rho_{n,i}:=X_{i,1}+\cdots+X_{i,n}.
$$
Since $\sigma>0$ we get $q:=\mathds{P}(X_{1,1}=0)<1$ and thus
$$
\sum_{n=1}^\infty\mathds{P}(\mathcal{D}_n^c) %
=\sum_{n=1}^{\infty} (1-(1-q^n)^n) %
\leq\sum_{n=1}^\infty nq^n<\infty.
$$
By the first Borel-Cantelli lemma, a.s.\ for $n\gg1$, one can define the
$n\times n$ matrix $M$ by
$$
M_{i,j}:=\frac{X_{i,j}}{\rho_{n,i}}.
$$
The matrix $M$ is Markov since its entries belong to $[0,1]$ and each row sums
up to $1$. We have $M=DX$ where $X:=(X_{i,j})_{1\leq i,j\leq n}$ and $D$ is
the $n\times n$ diagonal matrix defined by
$$
D_{i,i}:=\frac{1}{\rho_{n,i}}.
$$
We may define $M$ and $D$ for all $n\geq1$ by setting, when $\rho_{n,i}=0$,
$M_{i,j}=\delta_{i,j}$ for all $1\leq j\leq n$ and $D_{i,i}=1$. The matrix $M$
has equally distributed dependent entries. However, the rows of $M$ are
i.i.d.\ and follow an exchangeable law on $\mathds{R}^n$ supported by the
simplex
$$
\Lambda_n:=\{(p_1,\ldots,p_n)\in[0,1]^n:p_1+\cdots+p_n=1\}.
$$
From now on, we set $m=1$. This is actually no loss of generality since the
law of the random matrix $M$ is invariant under the linear scaling $t\to
t\,X_{i,j}$ for any $t>0$. Since $\sigma<\infty$, the uniform law of large
numbers of Bai and Yin \cite[lem. 2]{MR1235416} states that a.s.\
\begin{equation}\label{eq:ULLN}
\max_{1\leq i\leq n}\ABS{\rho_{n,i}-n}=o(n).
\end{equation}
This suggests that $\sqrt{n}M$ is approximately equal to $n^{-1/2}X$ for
$n\gg1$. One can then expect that \eqref{eq:QCLT} and \eqref{eq:CLT} hold for
$\sqrt{n}M$. Our work shows that this heuristics is valid. There is however a
complexity gap between \eqref{eq:QCLT} and \eqref{eq:CLT}, due to the fact
that for nonnormal operators such as $M$, the eigenvalues are less stable than
the singular values under perturbations, see e.g.\ the book \cite{MR2155029}.
Our first result below constitutes the analog of the universal
Marchenko-Pastur theorem \eqref{eq:QCLT} for $\sqrt{n}M$, and generalizes the
result of the same kind obtained in \cite{djalil-dme} in the case where
$X_{1,1}$ follows an exponential law.

\begin{thm}[Quartercircular law theorem]\label{th:singular}
  We have a.s.\
  $$
  \nu_{\sqrt{n}M} %
  \underset{n\to\infty}{\overset{\mathscr{C}_b}{\longrightarrow}}%
  \mathcal{Q}_\sigma.
  $$
\end{thm}

Our second result provides some estimates on the largest singular values and
eigenvalues.

\begin{thm}[Extremes]\label{th:localization}
  We have $\lambda_1(M)=1$. Moreover, if
  $\mathds{E}(|X_{1,1}|^4)<\infty$ then a.s.\
  $$
  \lim_{n\to\infty}s_1(M)= 1
  \quad\text{and}\quad
  \lim_{n\to\infty}
  s_2(\sqrt{n}M)= 2\sigma
  \quad\text{while}\quad %
  \varlimsup_{n\to\infty}|\lambda_2(\sqrt{n}M)|\leq 2\sigma.
  $$
\end{thm}
  
Our third result below is the analogue of \eqref{eq:CLT} for our random Markov
matrices. When $X_{1,1}$ follows the exponential distribution of unit mean,
theorem \ref{th:circular} is exactly the circular law theorem for the
Dirichlet Markov Ensemble conjectured in \cite{djalil-dme,MR2549497}. Note
that we provide, probably for the first time, an almost sure circular law
theorem for a matrix model with dependent entries under a finite positive
variance assumption.

\begin{thm}[Circular law theorem]\label{th:circular}
  If $X_{1,1}$ has a bounded density then a.s.\ 
  $$
  \mu_{\sqrt{n}M}
  \underset{n\to\infty}{\overset{\mathscr{C}_b}{\longrightarrow}}%
  \mathcal{U}_\sigma.
  $$
\end{thm}

The proof of theorem \ref{th:circular} is crucially based on the following
estimate on the norm of the resolvent of $\sqrt{n}M$.
It is the analogue of \eqref{eq:sn-exp-pol} for our random Markov matrices.

\begin{thm}[Smallest singular value]\label{th:least-singular} 
  If $X_{1,1}$ has a bounded density then for every $a,C>0$ there exists $b>0$
  such that for any $z\in\mathds{C}$ with $\ABS{z}\leq C$, for $n\gg1$,
  $$
  \mathds{P}(s_n(\sqrt{n}M-zI)\leq n^{-b})\leq n^{-a}.
  $$
  In particular, for some $b>0$ which may depend on $C$, a.s.\ for $n\gg1$,
  the matrix $\sqrt{n}M-zI$ is invertible with $s_n(\sqrt{n}M-zI) \geq
  n^{-b}$.
\end{thm}

The proofs of theorems
\ref{th:singular}-\ref{th:localization}-\ref{th:circular}-\ref{th:least-singular}
are given in sections
\ref{se:singular}-\ref{se:localization}-\ref{se:circular}-\ref{se:least-singular}
respectively. These proofs make heavy use of lemmas given in the appendices
\ref{se:logpot}-\ref{se:spest}-\ref{se:adle}.

The matrix $M$ is the Markov kernel associated to the weighted oriented
complete graph with $n$ vertices with one loop per vertex, for which each edge
$i\to j$ has weight $X_{i,j}$. The skeleton of this kernel is an oriented
Erd\H{o}s-R\'enyi random graph where each edge exists independently of the
others with probability $1-q$. If $q=0$ then $M$ has a complete skeleton, is
aperiodic, and $1$ is the sole eigenvalue of unit module \cite{MR2209438}. The
nonoriented version of this graphical construction gives rise to random
reversible kernels for which a semicircular theorem is available
\cite{bordenave-caputo-chafai}. The bounded density assumption forces $q=0$.

Since $M$ is Markov, we have that for every integer $r\geq0$,
\begin{equation}\label{eq:mom}
  \int_{\mathds{C}}\!z^r\,\mu_M(dz) %
  = \frac{1}{n}\sum_{i=1}^n\lambda_i^r(M) %
  = \frac{1}{n}\sum_{i=1}^n p_M(r,i)
\end{equation}
where 
$$
p_M(r,i):=\sum_{\substack{1\leq i_1,\ldots,i_r\leq n\\i_1=i_r=i}} %
M_{i_1,i_2}\cdots M_{i_{r-1},i_1}
$$
is simply, conditional on $M$, the probability of a loop of length $r$ rooted
at $i$ for a Markov chain with transition kernel $M$. This provides a
probabilistic interpretation of the moments of the empirical spectral
distribution $\mu_M$ of $M$. The random Markov matrix $M$ is a \emph{random
  environment}. By combining theorem \ref{th:localization} with theorem
\ref{th:circular} and the identity \eqref{eq:mom}, we get that for every fixed
$r\geq0$, a.s.\
$$
\lim_{n\to\infty}n^{\frac{r}{2}}%
\left(\frac{1}{n}\sum_{i=1}^np_M(r,i)-\frac{1}{n}\right)=0.
$$

\subsection{Discussion and open questions}

The bounded density assumption in theorem \ref{th:circular} is only due to the
usage of theorem \ref{th:least-singular} in the proof. We believe that theorem
\ref{th:least-singular} (and thus \ref{th:circular}) is valid without this
assumption, but this is outside the scope of the present work, see remark
\ref{rm:singass} and figure \ref{fi:simus}. Our proof of theorem
\ref{th:circular} is inspired from the Tao and Vu proof of \eqref{eq:CLT}
based on Girko Hermitization, and allows actually to go beyond the circular
law, see remarks \ref{rm:CLT}-\ref{rm:beyond}. Concerning the extremes, by
theorems \ref{th:localization}-\ref{th:circular}, a.s.\
$$
\sigma\leq \varliminf_{n\to\infty}|\lambda_2(\sqrt{n}M)| \leq
\varlimsup_{n\to\infty}|\lambda_2(\sqrt{n}M)|\leq 2\sigma.
$$
Also, a.s.\ the ``spectral gap'' of $M$ is a.s.\ of order $1-O(n^{-1/2})$
(compare with the results of \cite{MR1972678}). Note that in contrast with
\eqref{eq:silara}, we have from theorems
\ref{th:singular}-\ref{th:localization}, a.s.
$$
\lim_{n\to\infty} \frac{s_1(M)}{|\lambda_1(M)|}=1.
$$
Numerical simulations suggest that if $\mathds{E}(|X_{1,1}|^4)<\infty$ then
a.s.\
$$
\lim_{n\to\infty}|\lambda_2(\sqrt{n}M)|=\sigma
\quad\text{and thus}\quad
\lim_{n\to\infty}\frac{s_2(M)}{|\lambda_2(M)|}=2.
$$
Unfortunately, our proof of theorem \ref{th:localization} is too perturbative
to extract this result. Following \cite[fig. 2]{djalil-dme}, one can also ask
if the phase $\mathrm{Phase}(\lambda_2(M))=\lambda_2(M)/\ABS{\lambda_2(M)}$
converges in distribution to the uniform law on $[0,2\pi]$ as $n\to\infty$.
Another interesting problem concerns the behavior of $\max_{2\leq k\leq
  n}\mathfrak{Re}(\lambda_k(\sqrt{n}M))$ and the fluctuation as $n\to\infty$
of the extremal singular values and eigenvalues of $\sqrt{n}M$, in particular
the fluctuation of $\lambda_2(\sqrt{n}M)$.

Classical results on the connectivity of Erd\H{o}s-R\'enyi random graphs
\cite{MR0125031,MR1864966} imply that a.s.\ for $n\gg1$ the Markov matrix $M$
is irreducible. Hence, a.s.\ for $n\gg1$, the Markov matrix $M$ admits a
unique invariant probability measure $\kappa$. If $\kappa$ is seen as a row
vector in $\Lambda_n$ then we have $\kappa M=\kappa$. Let
$\Upsilon:=n^{-1}(\delta_1+\cdots+\delta_n)$ be the uniform law on
$\{1,\ldots,n\}$, which can be viewed as the vector
$n^{-1}(1,\ldots,1)\in\Lambda_n$. By denoting $\TV{\cdot}$ the total variation
(or $\ell^1$) distance on $\Lambda_n$, one can ask if a.s.\
$$
\lim_{n\to\infty}\TV{\kappa-\Upsilon}=0.
$$

Recall that the rows of $M$ are i.i.d.\ and follow an exchangeable law
$\eta_n$ on the simplex $\Lambda_n$. By ``exchangeable'' we mean that if
$Z\sim\eta_n$ then for every permutation $\pi$ of $\{1,\ldots,n\}$ the random
vector $(Z_{\pi(1)},\ldots,Z_{\pi(n)})$ follows also the law $\eta_n$. This
gives
$$
0=\mathrm{Var}(1)=\mathrm{Var}(Z_1+\cdots+Z_n)%
=n\mathrm{Var}(Z_1)+n(n-1)\mathrm{Cov}(Z_1,Z_2)
$$
and therefore $\mathrm{Cov}(Z_1,Z_2)=-(n-1)^{-1}\mathrm{Var}(Z_1)\leq0$. One
can ask if the results of theorems
\ref{th:singular}-\ref{th:localization}-\ref{th:circular}-\ref{th:least-singular}
remain essentially valid at least if $M$ is a real $n\times n$ random matrix
with i.i.d.\ rows such that for every $1\leq i,j\neq j'\leq n$,
$$
\mathds{E}(M_{i,j})=\frac{1}{n} %
\quad\text{and}\quad %
0<\mathrm{Var}(M_{i,j})=O(n^{-2}) %
\quad\text{and}\quad %
|\mathrm{Cov}(M_{i,j},M_{i,j'})|=O(n^{-3}).
$$
These rates in $n$ correspond to the Dirichlet Markov Ensemble for which
$X_{1,1}$ follows an exponential law and where $\eta_n$ is the Dirichlet law
$\mathcal{D}_n(1,\ldots,1)$ on the simplex $\Lambda_n$. Another interesting
problem is the spectral analysis of $M$ when the law of $X_{1,1}$ has heavy
tails, e.g.\ $X_{1,1}=V^{-\beta}$ with $2\beta>1$ and where $V$ is a uniform
random variable on $[0,1]$, see for instance \cite{bordenave-caputo-chafai-ii}
for the reversible case. A logical first step consists in the derivation of a
heavy tailed version of \eqref{eq:CLT} for $X$. This program is addressed in a
separate paper \cite{bordenave-caputo-chafai-heavygirko}. In the same spirit,
one may ask about the behavior of $\mu_{X\circ A}$ where ``$\circ$'' denotes
the Schur-Hadamard entrywise product and where $A$ is some prescribed profile
matrix.

\section{Proof of theorem \ref{th:singular}}\label{se:singular}

Let us start by an elementary observation. The second moment $\varsigma_n$ of
$\nu_{\sqrt{n}M}$ is given by
$$
\varsigma_n%
:=\int\!t^2\,\nu_{\sqrt{n}M}(dt)
=\frac{1}{n}\sum_{i=1}^ns_i(\sqrt{n}M)^2
=\mathrm{Tr}(MM^*)
=\sum_{i=1}^n\frac{X_{i,1}^2+\cdots+X_{i,n}^2}{(X_{i,1}+\cdots+X_{i,n})^2}.
$$
By using \eqref{eq:ULLN} together with the standard law of large numbers, we
get that a.s.\
\begin{equation}\label{eq:varbound}
  \varsigma_n \leq \frac{1}{n^2(1+o(1))^2}\sum_{i,j=1}^nX_{i,j}^2 %
  =(1+\sigma^2)+o(1) %
  =O(1).
\end{equation}
It follows by the Markov inequality that a.s.\ the sequence
$(\nu_{\sqrt{n}M)})_{n\geq1}$ is tight. However, we will not rely on tightness
and the Prohorov theorem in order to establish the convergence of
$(\nu_{\sqrt{n}M})_{n\geq1}$. We will use instead a perturbative argument
based on the special structure of $M$ and on \eqref{eq:ULLN}. Namely, since
$\sqrt{n}M=nDn^{-1/2}X$, we get from \eqref{eq:basic3}, for all $1\leq i\leq
n$,
\begin{equation}\label{eq:sisang}
  s_n(nD)s_i(n^{-1/2}X)
  \leq s_i(\sqrt{n}M)
  \leq s_1(nD)s_i(n^{-1/2}X).
\end{equation}
Additionally, we get from \eqref{eq:ULLN} that a.s.\
$$
\lim_{n\to\infty}\max_{1\leq i\leq n}|nD_{i,i}-1|=0 %
\quad\text{and}\quad %
\lim_{n\to\infty}\max_{1\leq i\leq n}|n^{-1}D_{i,i}^{-1}-1|=0.
$$
This gives that a.s.\
\begin{equation}\label{eq:sD}
  s_1(nD)=\max_{1\leq i\leq n}|nD_{i,i}|=1+o(1)
  \quad\text{and}\quad
  s_n(nD)=\min_{1\leq i\leq n}|nD_{i,i}|=1+o(1).
\end{equation}
From \eqref{eq:sisang}, \eqref{eq:sD}, and \eqref{eq:sn-exp-pol}, we get that
a.s.\ for $n\gg1$,
\begin{equation}\label{eq:isinv}
  s_n(\sqrt{n}M)>0
  \quad\text{and}\quad
  s_n(n^{-1/2}X)>0
\end{equation}
and from \eqref{eq:sisang} and \eqref{eq:sD} again we obtain that a.s.\ 
\begin{equation}\label{eq:difflog}
  \max_{1\leq i\leq n}
  \ABS{\log(s_i(\sqrt{n}M))-\log(s_i(n^{-1/2}X))}
  =o(1).
\end{equation}
Now, from \eqref{eq:QCLT}, and by denoting $\mathcal{L}_\sigma$ the image
probability measure of $\mathcal{Q}_\sigma$ by $\log(\cdot)$, a.s.\
$$
\frac{1}{n}\sum_{i=1}^n\delta_{\log(s_i(n^{-1/2}X))}
\underset{n\to\infty}{\overset{\mathscr{C}_b}{\longrightarrow}}%
\mathcal{L}_\sigma. 
$$
Next, using \eqref{eq:difflog} with  lemma \ref{le:perturb}
provides that a.s.\ 
$$
\frac{1}{n}\sum_{i=1}^n\delta_{\log(s_i(\sqrt{n}M))}
\underset{n\to\infty}{\overset{\mathscr{C}_b}{\longrightarrow}}%
\mathcal{L}_\sigma. 
$$
This implies by the change of variable $t\mapsto e^t$ that a.s.\
$$
\nu_{\sqrt{n}M}
=\frac{1}{n}\sum_{i=1}^n\delta_{s_i(\sqrt{n}M)}
\underset{n\to\infty}{\overset{\mathscr{C}_b}{\longrightarrow}}%
\mathcal{Q}_\sigma 
$$
which is the desired result. 

\begin{rem}[Alternative arguments] From \eqref{eq:sisang} and \eqref{eq:sD},
  a.s.\ for $n\gg1$, $s_k(M)=0$ iff $s_k(X)=0$, for all $1\leq k\leq n$, and
  thus $\nu_{\sqrt{n}M}(\{0\})=\nu_{n^{-1/2}X}(\{0\})$, and hence the
  reasoning can avoid the usage of \eqref{eq:isinv}. Note also that
  \eqref{eq:isinv} is automatically satisfied when $X_{1,1}$ is absolutely
  continuous since the set of singular matrices has zero Lebesgue measure. On
  the other hand, it is also worthwhile to mention that if
  $\mathds{E}(|X_{1,1}|^4)<\infty$ then one can obtain the desired result
  without using \eqref{eq:isinv}, by using lemma \ref{le:perturb} together
  with \eqref{eq:deltasi}, and this reasoning was already used by Aubrun for a
  slightly different model \cite{MR2280648}.
\end{rem}

\section{Proof of theorem \ref{th:localization}}\label{se:localization}

Let us define the $n\times n$ deterministic matrix
$S:=\mathds{E}(X)=(1,\ldots,1)^\top(1,\ldots,1)$. The random matrix
$n^{-1/2}(X-S)$ has i.i.d.\ centered entries with finite positive variance
$\sigma^2$ and finite fourth moment, and consequently, by \eqref{eq:s1X},
a.s.\
$$
s_1(n^{-1/2}(X-S))=2\sigma+o(1).
$$
Now, since $\mathrm{rank}(n^{-1/2}S)=1$ we have, by lemma \ref{le:thompson},
$$
s_2(n^{-1/2}X)\leq s_1(n^{-1/2}(X-S))
$$
and therefore, a.s.\ 
\begin{equation}\label{eq:s2X}
  s_2(n^{-1/2}X)\leq 2\sigma+o(1).
\end{equation}
By combining \eqref{eq:sisang} with \eqref{eq:sD} and \eqref{eq:s2X} we get
that a.s.\ 
\begin{equation}\label{eq:deltasi}
  \max_{2\leq i\leq n}|s_i(\sqrt{n}M)-s_i(n^{-1/2}X)|=o(1).
\end{equation}
In particular, this gives from \eqref{eq:s2X} that a.s.\ 
\begin{equation}\label{eq:s2Msup}
  s_2(\sqrt{n}M)\leq 2\sigma+o(1).
\end{equation}
From theorem \ref{th:singular}, since $\mathcal{Q}_\sigma$ is supported by
$[0,2\sigma]$, we get by using \eqref{eq:s2Msup} that a.s.\
\begin{equation}\label{eq:s2}
  \lim_{n\to\infty}s_2(\sqrt{n}M)=2\sigma.
\end{equation}
Next, since $M$ is a Markov matrix, it is known that
\begin{equation}\label{eq:la1}
  \lambda_1(M)=1.
\end{equation}
Let us briefly recall the proof. If $u:=(1,\ldots,1)^\top$ then $Mu=u$ and
thus $1$ is an eigenvalue of $M$. Next, let $\lambda\in\mathds{C}$ be an
eigenvalue of $M$ and let $x\in\mathds{C}^n$ be such that $x\neq0$ and
$Mx=\lambda x$. There exists $1\leq i\leq n$ such that
$|x_i|=\max\{|x_1|,\ldots,|x_n|\}$. Since $|x_i|\neq0$ and
$$
|\lambda||x_i|\leq \sum_{j=1}^nM_{i,j}|x_j|\leq
|x_i|\sum_{j=1}^nM_{i,j}=|x_i|,
$$
we get $|\lambda|\leq1$, which implies \eqref{eq:la1}. Let us show now that
a.s.
\begin{equation}\label{eq:s1}
  \lim_{n\to\infty}s_1(M)=1
\end{equation}
Let $S$ be as in the proof of theorem \ref{th:singular}. From
\eqref{eq:basic0} and \eqref{eq:s1X}, we get a.s.
\begin{align*}
  s_1(\sqrt{n}M)
  &\leq s_1(nD)s_1(n^{-1/2}X) \\
  &\leq s_1(nD)s_1(n^{-1/2}(X-S)+n^{-1/2}S) \\
  &\leq s_1(nD)(s_1(n^{-1/2}(X-S))+s_1(n^{-1/2}S)) \\
  &=(1+o(1))(2\sigma+o(1)+n^{1/2})
\end{align*}
which gives $\varlimsup_{n\to\infty}s_1(M)\leq1$ a.s. On the other hand, from
\eqref{eq:weyl0} and \eqref{eq:la1} we get $s_1(M)\geq\ABS{\lambda_1(M)}=1$,
which gives \eqref{eq:s1}. It remains to establish that a.s.\
$$
\varlimsup_{n\to\infty}|\lambda_2(\sqrt{n}M)|\leq 2\sigma.
$$
Indeed, from \eqref{eq:weyl0} we get for every non null $n\times n$ complex
matrix $A$,
$$
|\lambda_2(A)|\leq
\frac{s_1(A)s_2(A)}{|\lambda_1(A)|}.
$$
With $A=\sqrt{n}M$ and by using (\ref{eq:s2}-\ref{eq:la1}-\ref{eq:s1}), we
obtain that a.s.\
$$
\varlimsup_{n\to\infty} |\lambda_2(\sqrt{n}M)| %
\leq \varlimsup_{n\to\infty}s_2(\sqrt{n}M)\varlimsup_{n\to\infty}s_1(M)%
= 2\sigma.
$$

\section{Proof of theorem \ref{th:circular}}\label{se:circular}

Let us start by observing that from the Weyl inequality \eqref{eq:weyl3}
and \eqref{eq:varbound}, a.s.\
$$
\int_{\mathds{C}}\!|z|^2\,\mu_{\sqrt{n}M}(dt) 
\leq \int_0^\infty\!t^2\,\nu_{\sqrt{n}M}(dt) 
= O(1).
$$
This shows via the Markov inequality that a.s.\ the sequence
$(\mu_{\sqrt{n}M})_{n\geq1}$ is tight. However, we will not rely directly on
this tightness and the Prohorov theorem in order to establish the convergence
of $(\mu_{\sqrt{n}M})_{n\geq1}$. We will use instead the Girko Hermitization
of lemma \ref{le:girko}. We know, from the work of Dozier and Silverstein
\cite{MR2322123}, that for all $z\in\mathds{C}$, there exists a probability
measure $\nu_z$ on $[0,\infty)$ such that a.s.\
$(\nu_{n^{-1/2}X-zI})_{n\geq1}$ converges weakly to $\nu_z$. Moreover,
following e.g.\ Pan and Zhou \cite[lem. 3]{1687963}, one can check that for
all $z\in\mathds{C}$,
$$
U_{\mathcal{U}_\sigma}(z)=-\int_0^\infty\!\log(t)\,\nu_z(dt).
$$
To prove that a.s.\ $(\mu_{\sqrt{n}M})_{n\geq1}$ tends weakly to
$\mathcal{U}_\sigma$, we start from the decomposition
$$
\sqrt{n}M-zI
=nDW
\quad\text{where}\quad
W:=n^{-1/2}X-zn^{-1}D^{-1}.
$$
By using \eqref{eq:basic1}, \eqref{eq:ULLN}, and lemma \ref{le:perturb}, we
obtain that for a.a.\ $z\in\mathds{C}$, a.s.\ 
$$
\nu_W %
\underset{n\to\infty}{\overset{\mathscr{C}_b}{\longrightarrow}} %
\nu_z.
$$
Now, arguing as in the proof of theorem \ref{th:singular}, it follows that for
all $z\in\mathds{C}$, a.s.\ 
$$
\nu_{\sqrt{n}M-zI} %
\underset{n\to\infty}{\overset{\mathscr{C}_b}{\longrightarrow}} %
\nu_z.
$$
Suppose for the moment that for a.a.\ $z\in\mathds{C}$, a.s.\ the function
$\log(\cdot)$ is uniformly integrable for $(\nu_{\sqrt{n}M-zI})_{n\geq1}$. Let
$\mathcal{P}(\mathds{C})$ by as in section \ref{se:logpot}. Lemma
\ref{le:girko} implies that there exists $\mu\in\mathcal{P}(\mathds{C})$ such
that a.s.\
$$
\mu_{\sqrt{n}M} %
\underset{n\to\infty}{\overset{\mathscr{C}_b}{\longrightarrow}} %
\mu
\quad\text{and}\quad
U_\mu=U_{\mathcal{U}_\sigma}\text{ a.e.}
$$
where $U_\mu$ is the logarithmic potential of $\mu$ as defined in section
\ref{se:logpot}. Now by lemma \ref{le:unicity}, we obtain
$\mu=\mathcal{U}_\sigma$, which is the desired result. It thus remains to show
that for a.a.\ $z\in\mathds{C}$, a.s.\ the function $\log(\cdot)$ is uniformly
integrable for $(\nu_{\sqrt{n}M-zI})_{n\geq1}$. For every $z\in\mathds{C}$,
a.s.\ by the Cauchy-Schwarz inequality, for all $t\geq1$, and $n\gg1$,
$$
\PAR{\int_t^\infty\!\log(s)\,\nu_{\sqrt{n}M-zI}(ds)}^2
\leq \nu_{\sqrt{n}M-zI}([t,\infty))\int_0^\infty\!s^2\,\nu_{\sqrt{n}M-zI}(ds).
$$
Now the Markov inequality and \eqref{eq:varbound} give that for all
$z\in\mathds{C}$, a.s.\ for all $t\geq1$
$$
\int_t^\infty\!\log(s)\,\nu_{\sqrt{n}M-zI}(ds) \leq \frac{O(1)}{t^2}
$$
where the $O(1)$ is uniform in $t$. Consequently, for all $z\in\mathds{C}$,
a.s.\
$$
\lim_{t\to\infty}\varlimsup_{n\to\infty}\int_t^\infty\!\log(s)\,\nu_{\sqrt{n}M-zI}(ds)=0.
$$
This means that for all $z\in\mathds{C}$, a.s.\ the function
$\mathds{1}_{[1,\infty)}\log(\cdot)$ is uniformly integrable for
$(\nu_{\sqrt{n}M-zI})_{n\geq1}$. It remains to show that for all
$z\in\mathds{C}$, a.s. the function $\mathds{1}_{(0,1)}\log(\cdot)$ is
uniformly integrable for $(\nu_{\sqrt{n}M-zI})_{n\geq1}$. This is equivalent
to show that for all $z \in \mathds{C}$, a.s.
$$
\lim_{\delta \to 0} \varlimsup_{n\to\infty} 
\int_0^\delta\! -\log(s)\,\nu_{\sqrt n  M - z I}(ds) = 0.
$$
For convenience, we fix $z\in\mathds{C}$ and set $s_i := s_i (\sqrt{n}M - zI)$ for
all $1 \leq i \leq n$. Now we write
\begin{align*}
  -\int_0^\delta\!\log(t)\,\nu_{\sqrt n M - z I}(dt) %
  &= \frac{1}{n}\!\!\sum_{i=0}^{\FLOOR{2 n^{0.99}}}%
  \!\!\!\!\mathds{1}_{(0,\delta)}(s_{n-i})\log(s_{n-i}^{-1}) %
  +\frac{1}{n}\!\!\sum_{i=\FLOOR{2 n^{0.99}} +1 }^{ n-1 }
  \!\!\!\!\!\!\mathds{1}_{(0,\delta)}(s_{n-i})\log(s_{n-i}^{-1}) \\
  & \leq \frac{\log(s_{n}^{-1})}{n}\!\!\sum_{i=0}^{\FLOOR{2 n^{0.99}}}%
  \!\!\!\!\mathds{1}_{(0,\delta)}(s_{n-i}) %
  + \frac{1}{n}\!\!\sum_{i=\FLOOR{2 n^{0.99}} +1 }^{ n-1 } %
  \!\!\!\!\mathds{1}_{(0,\delta)}(s_{n-i})\log (s_{n-i}^{-1}).
\end{align*}
From theorem \ref{th:least-singular} (here we need the bounded density
assumption) we get that a.s.\
$$
\lim_{\delta\to0}\varlimsup_{n\to\infty}
\frac{\log(s_{n}^{-1})}{n}\!\!\sum_{i=0}^{\FLOOR{2 n^{0.99}}}%
  \!\!\!\!\mathds{1}_{(0,\delta)}(s_{n-i})
=0
$$
and it thus remains to show that a.s.\
$$
\lim_{\delta \to 0} \varlimsup_{n\to\infty}\frac{1}{n} 
\sum_{i=\FLOOR{2 n^{0.99}} +1 }^{ n-1 }%
\mathds{1}_{(0,\delta)}(s_{n-i})\log(s_{n-i}^{-1}) = 0.
$$
This boils down to show that there exists $c_0>0$ such a.s.\ for $n\gg1$ and
$2 n^{0.99} \leq i \leq n-1$,
\begin{equation}\label{eq:domsg}
s_{n-i} \geq c_0 \frac i n.  
\end{equation}
To prove it, we adapt an argument due to Tao and Vu \cite{tao-vu-cirlaw-bis}.
We fix $2 n^{0.99} \leq i \leq n-1$ and we consider the matrix $M'$ formed by
the first $n- \CEIL{i/2}$ rows of
$$
\sqrt{n}(\sqrt{n}M-zI)=nDX-\sqrt{n}zI.
$$
By the Cauchy interlacing lemma \ref{le:cauchy}, we get
$$
n ^{-1/2} s'_{n-i} \leq  s_{n-i}
$$
where $s'_j := s_j (M')$ for all $1\leq j \leq n - \CEIL{i/2}$ are the
singular values of the rectangular matrix $M'$ in nonincreasing order. Next,
by the Tao and Vu negative moment lemma \ref{le:tvneg},
$$
s'^{-2}_1  + \cdots + s'^{-2}_{n -  \CEIL{i/2}} %
= \mathrm{dist}_1^{-2}+\cdots+\mathrm{dist}_{n -\CEIL{i/2}}^{-2}, 
$$
where $\mathrm{dist}_j$ is the distance from the $j^\text{th}$ row of $M'$ to $H_j$,
the subspace spanned by the other rows of $M'$. In particular, we have
\begin{equation}\label{eq:stodist}
  \frac i 2 s^{-2}_{n-i} \leq n \sum_{j=1}^{n-\CEIL{i/2}}\mathrm{dist}_{j}^{-2}. 
\end{equation}
Let $R_j$ be the $j^\text{th}$ row of $X$. 
Since the $j^\text{th}$ row of $M$ is $D_{j,j}R_j$, we deduce that
$$
\mathrm{dist}_{j} %
= \mathrm{dist}( n D_{j,j} R_j-z \sqrt{n}e_j, H_j) %
\geq n D_{j,j} \mathrm{dist}( R_j , \mathrm{span}(H_j, e_j))
$$
where $e_1,\ldots,e_n$ is the canonical basis of $\mathds{R}^n$.
Since $\mathrm{span}(H_j, e_j )$ is independent of $R_j$
and
$$
\mathrm{dim}(\mathrm{span}(H_j,e_j)) \leq n - \frac{i}{2} \leq n - n^{0.99},
$$
lemma \eqref{le:concdist} gives
$$
\sum_{n\gg1}\mathds{P}\PAR{\bigcup_{i=2n^{0.99}}^{n-1}\bigcup_{j=1}^{n-\CEIL{i/2}}
\BRA{\mathrm{dist}(R_j,\mathrm{span}(H_j,e_j))\leq\frac{\sigma\sqrt{i}}{2\sqrt{2}}}
}<\infty
$$
(note that the exponential bound in lemma \ref{le:concdist} kills the
polynomial factor due to the union bound over $i,j$). Consequently, by the
first Borel-Cantelli lemma, we obtain that a.s.\ for $n\gg1$, all $2 n^{0.99}
\leq i \leq n-1$, and all $1 \leq j \leq n - \CEIL{i/2}$,
$$
\mathrm{dist}_j %
\geq nD_{j,j}\frac{\sigma\sqrt{i}}{2\sqrt{2}} %
= \sqrt{i}\frac{\sigma }{2\sqrt{2}}\frac{n}{\rho_{n,j}}.
$$
Now, the uniform law of large numbers \eqref{eq:ULLN} gives that a.s.
$$
\lim_{n\to\infty}\max_{ 1\leq j \leq n} \ABS{\frac{\rho_{n,j}}{n} - 1}= 0. 
$$
We deduce that a.s.\ for $n\gg1$, all $2 n^{0.99} \leq i \leq n-1$, and all $1
\leq j \leq n - \CEIL{i/2}$,
$$
\mathrm{dist}_j %
\geq \sqrt{i}\frac{\sigma}{4}
$$
Finally, from \eqref{eq:stodist} we get
$$
s^{2}_{n-i}\geq \frac {i^2}{n^2}\frac{\sigma^2}{32},
$$
and \eqref{eq:domsg} holds with $c_0 := \sigma/(4\sqrt{2})$.

\begin{rem}[Proof of the circular law \eqref{eq:CLT} and beyond]\label{rm:CLT}
  The same strategy allows a relatively short proof of \eqref{eq:CLT}. Indeed,
  the a.s.\ weak convergence of $(\mu_{n^{-1/2}X})_{n\geq1}$ to
  $\mathcal{U}_{\sigma}$ follows from the Girko Hermitization lemma
  \ref{le:girko} and the uniform integrability of $\log(\cdot)$ for
  $(\nu_{n^{-1/2}X-zI})_{n\geq1}$ as above using \eqref{eq:sn-exp-pol}. This
  direct strategy does not rely on the \emph{replacement principle} of Tao and
  Vu. The replacement principle allows a statement which is more general than
  \eqref{eq:CLT} involving two sequences of random matrices (the main result
  of \cite{tao-vu-cirlaw-bis} on universality). Our strategy allows to go
  beyond the circular law \eqref{eq:CLT}, by letting $\mathds{E}(X_{i,j})$
  possibly depend on $i,j,n$, provided that \eqref{eq:sn-exp-pol} and the
  result of Dozier and Silverstein \cite{MR2322123} hold. Set
  $A:=(\mathds{E}(X_{i,j}))_{1\leq i,j\leq n}$. If $\mathrm{Tr}(AA^*)$ is
  large enough for $n\gg1$, the limit is no longer the circular law, and can
  be interpreted by using free probability theory \cite{MR1929504}.
\end{rem}

\begin{rem}[Beyond the circular law]\label{rm:beyond} 
  It is likely that the Tao and Vu replacement principle
  \cite{tao-vu-cirlaw-bis} allows a universal statement for our random Markov
  matrices of the form $M=DX$, beyond the circular law, by letting
  $\mathds{E}(X_{i,j})$ possibly depend on $i,j,n$. This is however beyond the
  scope of the present work.
\end{rem}

\section{Proof of theorem \ref{th:least-singular}}
\label{se:least-singular}

Note that when $z=0$, one can get some $b$ immediately from (\ref{eq:basic2},
\ref{eq:sn-exp-pol}, \ref{eq:ULLN}). Thus, our problem is actually to deal
with $z\neq0$. Fix $a,C>0$ and $z\in\mathds{C}$ with $|z|\leq C$. We have
$$
\sqrt{n}M-zI=\sqrt{n}DY
\quad\text{where}\quad 
Y:=X-n^{-1/2}zD^{-1}.
$$
For an arbitrary $\delta_n>0$, let us define the event
$$
\mathcal{A}_n %
:=\bigcap_{i=1}^n\BRA{\ABS{\frac{\rho_{n,i}}{n}-1}\leq \delta_n}.
$$
By using the union bound and the Chebyshev inequality, we get
$\mathds{P}(\mathcal{A}_n^c)\leq \sigma^2 \delta_n^{-2}$. Now with $c > a/2$
and $\delta_n=n^c$ we obtain $\mathds{P}(\mathcal{A}_n^c)\leq n^{-a}$ for
$n\gg1$. Since we have 
$$
s_n(D)^{-1}=\max_{1\leq i\leq n}|\rho_{n,i}|,
$$
we get by \eqref{eq:basic2}, on the event $\mathcal{A}_n$, for $n\gg1$,
$$
\{s_n(\sqrt{n}M-zI)\leq t_n \} %
\subset \{\sqrt{n}s_n(D)s_n(Y)\leq t_n\} %
\subset \{s_n(Y) \leq \sqrt{n}t_n(1+n^c)\}
$$
for every $t_n>0$. Now, for every $b'>0$, one may select $b>0$ and set
$t_n=n^{-b}$ such that $\sqrt{n}t_n(1+n^c)\leq n^{-b'}$ for $n\gg1$. Thus, on
the event $\mathcal{A}_n$, for $n\gg1$,
$$
\mathcal{M}_n %
:= \{s_n(\sqrt{n}M-zI)\leq n^{-b} \} %
\subset \{s_n(Y) \leq n^{-b'}\} %
=:\mathcal{Y}_n.
$$
Consequently, for every $b'>0$ there exists $b>0$ such that for $n\gg1$,
$$
\mathds{P}(\mathcal{M}_n) %
=\mathds{P}(\mathcal{M}_n\cap\mathcal{A}_n) %
+\mathds{P}(\mathcal{M}_n\cap\mathcal{A}_n^c)  %
\leq \mathds{P}(\mathcal{Y}_n) %
+\mathds{P}(\mathcal{A}_n^c)  %
\leq \mathds{P}(\mathcal{Y}_n)+n^{-a}.
$$
The desired result follows if we show that for some $b'>0$ depending on $a,C$,
for $n\gg1$,
\begin{equation}\label{eq:sn2}
  \mathds{P}(\mathcal{Y}_n) = \mathds{P}(s_n(Y)\leq n^{-b'})\leq n^{-a}.
\end{equation}
Let us prove \eqref{eq:sn2}. At this point, it is very important to realize
that \eqref{eq:sn2} cannot follow form a perturbative argument based on
(\ref{eq:sn-exp-pol},\ref{eq:basic1},\ref{eq:ULLN}) since the operator norm of
the perturbation is much larger that the least singular value of the perturbed
matrix. We thus need a more refined argument. We have $Y=X-wD^{-1}$ with
$w:=n^{-1/2}z$. Let $A_w=A_{n^{-1/2}z}$ be as in lemma \ref{le:A}. For every
$1\leq k\leq n$, let $P_k$ be the $n\times n$ permutation matrix for the
transposition $(1,k)$. Note that $P_1=I$ and for every $1\leq k\leq n$, the
matrix $P_kA_wP_k$ is $n\times n$ lower triangular. For every column vector
$e_i$ of the canonical basis of $\mathds{R}^n$,
$$
(P_kAP_k)e_i=
\begin{cases}
  e_i & \text{if $i\neq k$}, \\
  e_k-w(e_1+\cdots+e_n) & \text{if $i=k$}.
\end{cases}
$$
Now, if $R_1,\ldots,R_n$ 
and $R_1',\ldots,R_n'$ are the rows of the matrices $X$ and $Y$ then
$$
Y %
=\begin{pmatrix}R_1' \\ \vdots \\ R_n'\end{pmatrix}
=\begin{pmatrix}R_1P_1A_wP_1 \\ \vdots \\ R_nP_nA_wP_n\end{pmatrix}.
$$
Define the vector space $R_{-i}':=\mathrm{span}\{R_j:j\neq i\}$ for every
$1\leq i\leq n$. From lemma \ref{le:rvdist},
$$
\min_{1\leq i\leq n}\mathrm{dist}(R_i',R_{-i}') 
\leq \sqrt{n}\,s_n(Y).
$$
Consequently, by the union bound, for any $u\geq0$,
$$
\mathds{P}(\sqrt{n}\,s_n(Y)\leq u) %
\leq n\max_{1\leq i\leq n}\mathds{P}(\mathrm{dist}(R_i',R_{-i}')%
\leq u).
$$
The law of $\mathrm{dist}(R_i',R_{-i}')$ does not depend on $i$. We take
$i=1$. Let $V'$ be a unit normal vector to $R_{-1}'$. Such a vector is not
unique, but we just pick one, and this defines a random variable on the unit
sphere $\mathds{S}^{n-1}:=\{x\in\mathds{C}^n:\NRM{x}_2=1\}$. Since $V' \in
R_{-1}'^\perp$ and $\NRM{V'}_2=1$,
$$
|R_1'\cdot V'|\leq \mathrm{dist}(R_1',R_{-1}').
$$
Let $\nu$ be the distribution of $V'$ on $\mathds{S}^{n-1}$. Since $V'$ and
$R_1'$ are independent, for any $u\geq0$,
$$
\mathds{P}(\mathrm{dist}(R_1',R_{-1}')\leq u)
\leq \mathds{P}(|R_1'\cdot V'|\leq u)
= \int_{\mathds{S}^{n-1}}\!\!\!\mathds{P}(|R_1'\cdot v'|\leq u)\,d\nu(v').
$$
Let us fix $v'\in\mathds{S}^{n-1}$. If $A_w,P_1=I,R_1$ are as above then
$$
R_1'\cdot v'=R_1\cdot v
\quad\text{where}\quad v:=P_1A_wP_1v'=A_wv'.
$$
Now, since $v'\in\mathds{S}^{n-1}$, lemma \ref{le:A} provides a constant $K>0$
such that for $n\gg1$,
$$
\NRM{v}_2 %
=\NRM{A_wv'}_2 %
\geq \min_{x\in\mathds{S}^{n-1}}\NRM{A_wx}_2 %
=s_n(A_w)\geq K^{-1}.
$$
But $\NRM{v}_2\geq K^{-1}$ implies $|v_{j}|^{-1}\leq K\sqrt{n}$ for some
$j\in\{1,\ldots,n\}$, and therefore
$$
|\mathfrak{Re}(v_{j})|^{-1}\leq K\sqrt{2n}
\quad\text{or}\quad
|\mathfrak{Im}(v_{j})|^{-1}\leq K\sqrt{2n}.
$$
Suppose for instance that we have $|\mathfrak{Re}(v_{j})|^{-1}\leq
K\sqrt{2n}$. We first observe that
$$
\mathds{P}(|R_1'\cdot v'|\leq u) 
=\mathds{P}(|R_1\cdot v|\leq u) 
\leq \mathds{P}(|\mathfrak{Re}(R_1\cdot v)|\leq u).
$$
The real random variable $\mathfrak{Re}(R_1\cdot v)$ is a sum of independent
real random variables and one of them is $X_{1,j}\mathfrak{Re}(v_{j})$, which
is absolutely continuous with a density bounded above by $BK\sqrt{2n}$ where
$B$ is the bound on the density of $X_{1,1}$. Consequently, by a basic
property of convolutions of probability measures, the real random variable
$\mathfrak{Re}(R_1\cdot v)$ is also absolutely continuous with a density
$\varphi$ bounded above by $BK\sqrt{2n}$, and therefore,
$$
\mathds{P}(|\mathfrak{Re}(R_1\cdot v)|\leq u)
= \int_{[-u,u]}\,\varphi(s)\,ds
\leq BK\sqrt{2n}2u.
$$  
To summarize, for $n\gg1$ and every $u\geq0$,
$$
\mathds{P}(\sqrt{n}s_n(Y)\leq u) \leq BK(2n)^{3/2}u.
$$
Lemma \ref{le:A} shows that the constant $K$ may be chosen depending on $C$
and not on $z$, and \eqref{eq:sn2} holds with $b'=d+1/2$ by taking $u=n^{-d}$
such that $BK(2n)^{3/2}n^{-d}\leq n^{-a}$ for $n\gg1$.

\begin{rem}[Assumptions]\label{rm:singass}
  Our proof of theorem \ref{th:least-singular} still works if the entries of
  $X$ are just independent and not necessarily i.i.d.\ provided that the
  densities are uniformly bounded and that \eqref{eq:ULLN} holds. The bounded
  density assumption allows to bound the small ball probability
  $\mathds{P}(\ABS{R_1\cdot v}\leq u)$ uniformly over $v$. If this assumption
  does not hold, then the small ball probability may depend on the additive
  structure of $v$, but the final result is probably still valid. A possible
  route, technical and uncertain, is to adapt the Tao and Vu proof of
  \eqref{eq:sn-exp-pol}. On the opposite side, if $X_{1,1}$ has a
  $\log$-concave density (e.g.\ exponential) then a finer bound might follow
  from a noncentered version of the results of Adamczak et al
  \cite{MR2441920}. Alternatively, if $X_{1,1}$ has sub-Gaussian or
  sub-exponential moments then one may also try to adapt the proof of Rudelson
  and Vershynin \cite{MR2407948} to the noncentered settings.
\end{rem}

\begin{rem}[Away from the limiting support]
  The derivation of an a.s.\ lower bound on $s_n(\sqrt{n}M-zI)$ is an easy
  task when $|z|>2\sigma$ and $\mathds{E}(|X_{1,1}|^4)<\infty$, without
  assuming that $X_{1,1}$ has a bounded density. Let us show for instance that
  for every $z\in\mathds{C}$, a.s.
  \begin{equation}\label{eq:sn3}
    s_n(\sqrt{n}M-zI) \geq |z|-2\sigma+o(1).
  \end{equation}
  This lower bound is meaningful only when $|z|>2\sigma$. For proving
  \eqref{eq:sn3}, we adopt a perturbative approach. Let us fix
  $z\in\mathds{C}$. By \eqref{eq:basic2} and \eqref{eq:ULLN} we get that a.s.
  \begin{equation}\label{eq:sn4}
    s_n(\sqrt{n}M-zI) \geq n^{-1}(1+o(1))\,s_n(\sqrt{n}X-zD^{-1}).
  \end{equation}
  Now we write, with $S=\mathbb{E}X=(1,\ldots,1)(1,\ldots,1)^\top$,
  $$
  \sqrt{n}X-zD^{-1} %
  = \sqrt{n}S-znI+W \quad\text{where}\quad W:=\sqrt{n}(X-S)+nzI-zD^{-1}.
  $$
  We observe that \eqref{eq:basic1} gives
  $$
  s_n(\sqrt{n}X-zD^{-1})\geq s_n(\sqrt{n}S-znI)-s_1(W).
  $$
  For the symmetric complex matrix $\sqrt{n}S-znI$ we have for any
  $z\in\mathds{C}$ and $n\gg1$,
  $$
  s_n(\sqrt{n}S-znI)=n\min(|z|,|\sqrt{n}-z|)=n|z|.
  $$
  On the other hand, since $\mathds{E}(|X_{1,1}|^4)<\infty$, by \eqref{eq:s1X}
  and \eqref{eq:ULLN}, a.s.\ for every $z\in\mathds{C}$,
  $$
  s_1(W) \leq s_1(\sqrt{n}(X-S))+s_1(nzI-zD^{-1}) =n(2\sigma+o(1))+|z|no(1).
  $$
  Putting all together, we have shown that a.s.\ for any $z\in\mathds{C}$,
  $$
  s_n(\sqrt{n}X-zD^{-1}) \geq n|z|(1-o(1))-n(2\sigma+o(1)).
  $$
  Combined with \eqref{eq:sn4}, this gives finally \eqref{eq:sn3}.
\end{rem}

\begin{rem}[Invertibility]
  Let $(A_n)_{n\geq1}$ be a sequence of complex random matrices where $A_n$ is
  $n\times n$ for every $n\geq1$, defined on a common probability space
  $(\Omega,\mathcal{A},\mathds{P})$. For every $\omega\in\Omega$, the set
  $\cup_{n\geq1}\{\lambda_1(A_n(\omega)),\ldots,\lambda_n(A_n(\omega))\}$ is
  at most countable and has thus zero Lebesgue measure. Therefore, for
  \textbf{all} $\omega\in\Omega$ and \textbf{a.a.}\ $z\in\mathds{C}$, we have
  $s_n(A_n(\omega)-zI)>0$ for all \textbf{all} $n\geq1$. Note that
  \eqref{eq:sn-exp-pol} and theorem \ref{th:least-singular} imply respectively
  that for $A_n=X$ or $A_n=M$, this holds for \textbf{all} $z\in\mathds{C}$,
  \textbf{a.s.}\ on $\omega$, and $n\boldsymbol{\gg}1$.
\end{rem}

\appendix

\section{Logarithmic potential and Hermitization}\label{se:logpot}

Let $\mathcal{P}(\mathds{C})$ be the set of probability measures on
$\mathds{C}$ which integrate $\log\ABS{\cdot}$ in a neighborhood of infinity.
For every $\mu\in\mathcal{P}(\mathds{C})$, the \emph{logarithmic potential}
$U_\mu$ of $\mu$ on $\mathds{C}$ is the function
$U_\mu:\mathds{C}\to(-\infty,+\infty]$ defined for every $z\in\mathds{C}$ by
\begin{equation}\label{eq:logpot}
  U_\mu(z)=-\int_{\mathds{C}}\!\log|z-z'|\,\mu(dz')
  =-(\log\ABS{\cdot}*\mu)(z).
\end{equation}
For instance, for the circular law $\mathcal{U}_1$ of density
$\pi^{-1}\mathds{1}_{\{z\in\mathds{C}:|z|\leq1\}}$, we have, for every
$z\in\mathds{C}$,
$$
U_{\mathcal{U}_1}(z)=
\begin{cases}
  -\log|z| & \text{if $|z|>1$}, \\
  \frac{1}{2}(1-|z|^2) & \text{if $|z|\leq1$},
\end{cases}
$$
see e.g.\ \cite{MR1485778}. Let $\mathcal{D}'(\mathds{C})$ be the set of
Schwartz-Sobolev distributions (generalized functions). Since
$\log\ABS{\cdot}$ is Lebesgue locally integrable on $\mathds{C}$, one can
check by using the Fubini theorem that $U_\mu$ is Lebesgue locally integrable
on $\mathds{C}$. In particular, $U_\mu<\infty$ a.e.\ and
$U_\mu\in\mathcal{D}'(\mathds{C})$. Since $\log\ABS{\cdot}$ is the fundamental
solution of the Laplace equation in $\mathds{C}$, we have, in
$\mathcal{D}'(\mathds{C})$,
\begin{equation}\label{eq:lap}
  \Delta U_\mu=-2\pi\mu.
\end{equation}
This means that for every smooth and compactly supported ``test function''
$\varphi:\mathds{C}\to\mathds{R}$,
$$
\int_{\mathds{C}}\!\Delta\varphi(z)U_\mu(z)\,dz
=-2\pi\int_{\mathds{C}}\!\varphi(z)\,\mu(dz).
$$

\begin{lem}[Unicity]\label{le:unicity}
  For every $\mu,\nu\in\mathcal{P}(\mathds{C})$, if $U_\mu=U_\nu$ a.e.\ then
  $\mu=\nu$.
\end{lem}

\begin{proof}
  Since $U_\mu=U_\nu$ in $\mathcal{D}'(\mathds{C})$, we get $\Delta
  U_\mu=\Delta U_\nu$ in $\mathcal{D}'(\mathds{C})$. Now \eqref{eq:lap} gives
  $\mu=\nu$ in $\mathcal{D}'(\mathds{C})$, and thus $\mu=\nu$ as measures
  since $\mu$ and $\nu$ are Radon measures.
\end{proof}

If $A$ is an $n\times n$ complex matrix and $P_A(z):=\det(A-zI)$ is its
characteristic polynomial,
$$
U_{\mu_A}(z)
=-\int_{\mathds{C}}\!\log\ABS{z'-z}\,\mu_A(dz')
=-\frac{1}{n}\log\ABS{\det(A-zI)}
=-\frac{1}{n}\log\ABS{P_A(z)}
$$
for every $z\in\mathds{C}\setminus\{\lambda_1(A),\ldots,\lambda_n(A)\}$. We
have also the alternative expression 
\begin{equation}\label{eq:UESD}
  U_{\mu_A}(z)
  =-\frac{1}{n}\log\det(\sqrt{(A-zI)(A-zI)^*})
  =-\int_0^\infty\!\log(t)\,\nu_{A-zI}(dt).
\end{equation}

The identity \eqref{eq:UESD} bridges the eigenvalues with the singular values,
and is at the heart of the following lemma, which allows to deduce the
convergence of $\mu_A$ from the one of $\nu_{A-zI}$. The strength of this
Hermitization lies in the fact that in contrary to the eigenvalues, one can
control the singular values with the entries of the matrix. The price payed
here is the introduction of the auxiliary variable $z$ and the uniform
integrability. We recall that on a Borel measurable space $(E,\mathcal{E})$,
we say that a Borel function $f:E\to\mathds{R}$ is \emph{uniformly integrable} 
for a sequence of probability measures $(\eta_n)_{n\geq1}$ on $E$ when
\begin{equation}
  \lim_{t\to\infty}\varlimsup_{n\to\infty}\int_{\{|f|>t\}}\!|f|\,d\eta_n=0.
\end{equation}
We will use this property as follows: if $(\eta_n)_{n\geq1}$ converges weakly
to $\eta$ and $f$ is continuous and uniformly integrable for
$(\eta_n)_{n\geq1}$ then $f$ is $\eta$-integrable and
$\lim_{n\to\infty}\int\!f\,d\eta_n=\int\!f\,\eta$. The idea of using
Hermitization goes back at least to Girko \cite{MR1080966}. However, the
proofs of lemmas \ref{le:girko} and \ref{le:unimaj} below are inspired from
the approach of Tao and Vu \cite{tao-vu-cirlaw-bis}.

\begin{lem}[Girko Hermitization]\label{le:girko}
  Let $(A_n)_{n\geq1}$ be a sequence of complex random matrices where $A_n$ is
  $n\times n$ for every $n\geq1$, defined on a common probability space.
  Suppose that for a.a.\ $z\in\mathds{C}$, there exists a probability measure
  $\nu_z$ on $[0,\infty)$ such that a.s.\
  \begin{itemize}
  \item[(i)] $(\nu_{A_n-zI})_{n\geq1}$ converges weakly to $\nu_z$ as $n\to\infty$
  \item[(ii)] $\log(\cdot)$ is uniformly integrable for
    $\PAR{\nu_{A_n-zI}}_{n\geq1}$
  \end{itemize}
  Then there exists a probability measure $\mu\in\mathcal{P}(\mathds{C})$ such
  that
  \begin{itemize}
  \item[(j)] a.s.\ $(\mu_{A_n})_{n\geq1}$ converges weakly to $\mu$ as
    $n\to\infty$
  \item[(jj)] for a.a.\ $z\in\mathds{C}$,
    $$
    U_\mu(z)=-\int_0^\infty\!\log(t)\,\nu_z(dt).
    $$
  \end{itemize}
  Moreover, if $(A_n)_{n\geq1}$ is deterministic, then the statements hold
  without the ``a.s.''
\end{lem}

\begin{proof}
  Let $z$ and $\omega$ be such that (i-ii) hold. For every $1\leq k\leq n$,
  define 
  $$
  a_{n,k}:=|\lambda_k(A_n-zI)|
  \quad\text{and}\quad 
  b_{n,k}:=s_k(A_n-zI)
  $$
  and set $\nu:=\nu_z$. Note that $\mu_{A_n-zI}=\mu_{A_n}*\delta_{-z}$. Thanks
  to the Weyl inequalities \eqref{eq:weyl0} and to the assumptions (i-ii), one
  can use lemma \ref{le:unimaj} below, which gives that $(\mu_{A_n})_{n\geq1}$
  is tight, that $\log\ABS{z-\cdot}$ is uniformly integrable for
  $(\mu_{A_n})_{n\geq1}$, and that
  $$
  \lim_{n\to\infty}U_{\mu_{A_n}}(z)=-\int_0^\infty\!\log(t)\,\nu_z(dt)=:U(z).
  $$
  Consequently, a.s.\ $\mu\in\mathcal{P}(\mathds{C})$ and $U_\mu=U$ a.e. for
  every adherence value $\mu$ of $(\mu_{A_n})_{n\geq1}$. Now, since $U$ does
  not depend on $\mu$, by lemma \ref{le:unicity}, a.s.\
  $\PAR{\mu_{A_n}}_{n\geq1}$ has a unique adherence value $\mu$, and since
  $(\mu_n)_{n\geq1}$ is tight, $(\mu_{A_n})_{n\geq1}$ converges weakly to
  $\mu$ by the Prohorov theorem. Finally, by \eqref{eq:lap}, $\mu$ is
  deterministic since $U$ is deterministic, and (j-jj) hold.
\end{proof}

The following lemma is in a way the skeleton of the Girko Hermitization of
lemma \ref{le:girko}. It states essentially a propagation of a uniform
logarithmic integrability for a couple of triangular arrays, provided that a
logarithmic majorization holds between the arrays.

\begin{lem}[Logarithmic majorization and uniform integrability]\label{le:unimaj}
  Let $(a_{n,k})_{1\leq k\leq n}$ and $(b_{n,k})_{1\leq k\leq n}$ be two
  triangular arrays in $[0,\infty)$. Define the discrete probability measures
  $$
  \mu_n:=\frac{1}{n}\sum_{k=1}^n\delta_{a_{n,k}}
  \quad\text{and}\quad
  \nu_n:=\frac{1}{n}\sum_{k=1}^n\delta_{b_{n,k}}.
  $$
  If the following properties hold 
  \begin{itemize}
  \item[(i)] $a_{n,1}\geq\cdots\geq a_{n,n}$ and $b_{n,1}\geq\cdots\geq
    b_{n,n}$ for $n\gg1$,
  \item[(ii)] $\prod_{i=1}^k a_{n,i} \leq \prod_{i=1}^k b_{n,i}$ for every
    $1\leq k\leq n$ for $n\gg1$,
  \item[(iii)] $\prod_{i=k}^n b_{n,i} \leq \prod_{i=k}^n a_{n,i}$ for every
    $1\leq k\leq n$ for $n\gg1$,
  \item[(iv)] $(\nu_n)_{n\geq1}$ converges weakly to some probability measure
    $\nu$ as $n\to\infty$,
  \item[(v)] $\log(\cdot)$ is uniformly integrable for $(\nu_n)_{n\geq1}$,
  \end{itemize}
  then 
  \begin{itemize}
  \item[(j)] $(\mu_n)_{n\geq1}$ is tight,
  \item[(jj)] $\log(\cdot)$ is uniformly integrable for $(\mu_n)_{n\geq1}$,
  \item[(jjj)] we have, as $n\to\infty$,
    $$
    \int_0^\infty\!\log(t)\,\mu_n(dt)
    =\int_0^\infty\!\log(t)\,\nu_n(dt)\to\int_0^\infty\!\log(t)\,\nu(dt),
    $$
  \end{itemize}
  and in particular, for every adherence value $\mu$ of $(\mu_n)_{n\geq1}$,
  $$
  \int_0^\infty\!\log(t)\,\mu(dt)=\int_0^\infty\!\log(t)\,\nu(dt).
  $$
\end{lem}

\begin{proof}
  \textbf{Proof of (jjj)}. From the logarithmic majorizations (ii-iii) we get,
  for $n\gg1$,
  $$
  \prod_{k=1}^na_{n,k}=\prod_{k=1}^nb_{n,k},
  $$
  and (v) gives $b_{n,k}>0$ and $a_{n,k}>0$ for every $1\leq k\leq n$ and
  $n\gg1$. Now, (iv-v) give
  \begin{align*}
  \int_0^\infty\!\!\!\!\log(t)\,\mu_n(dt)
  &=\frac{1}{n}\log\prod_{k=1}^na_{n,k} \\
  &=\frac{1}{n}\log\prod_{k=1}^nb_{n,k} \\
  &=\int_0^\infty\!\!\!\!\log(t)\,\nu_n(dt) 
  \to \int_0^\infty\!\!\!\!\log(t)\,\nu(dt).
  \end{align*}
  \textbf{Proof of (j)}. From (ii) and (v) we get
  $$
  \sup_{1\leq k\leq n}\sum_{i=1}^k\log(a_{n,i})
  \leq 
  \sup_{1\leq k\leq n}\sum_{i=1}^k\log(b_{n,i})
  \quad\text{and}\quad 
  C:=\sup_{n\geq1}\int_0^\infty\!|\log(s)|\,\nu_n(ds)<\infty
  $$
  respectively. Now the tightness of $(\mu_n)_{n\geq1}$ follows from
  \begin{equation}\label{eq:munu}
    \int_1^\infty\!\log(s)\,\mu_n(ds)
    \leq \int_1^\infty\!\log(s)\,\nu_n(ds)\leq C.
  \end{equation}
  \textbf{Proof of (jj).} We start with the uniform integrability in the
  neighborhood of infinity. Let us show that for $n\gg1$, for any
  $\varepsilon>0$ there exists $t\geq1$ such that
  \begin{equation}\label{eq:unifconvmu}
    \int_t^\infty\!\log(s)\,\mu_n(ds) <\varepsilon.
  \end{equation}
  If $\nu((1,\infty))=0$ then (iv) implies
  $$
  \int_1^\infty\!\log(t)\,\nu_n(dt)<\varepsilon
  $$
  for $n\gg1$ and \eqref{eq:unifconvmu} follows then from \eqref{eq:munu}. If
  otherwise $\nu((1,\infty))>0$ then 
  $$
  c:=\int_1^\infty\!\log(t)\,\nu(dt)>0
  $$
  and one can assume that $\varepsilon<c$. Let us show that there exists a
  sequence of integers $(k_n)_{n\geq1}$ such that
  $\lim_{n\to\infty}k_n/n\to\sigma>0$ and for $n\gg1$,
  \begin{equation}\label{eq:defkn}
    \sup_{1\leq k\leq k_n}\frac{1}{n}\sum_{i=1}^{k_n}\log(b_{n,i})<\varepsilon.
  \end{equation}
  For $0 < \varepsilon/2 < c$, let $t$ be the infimum over all $s>1$ such that
  $$
  \int_s^\infty\!\log(u)\,\nu(du) < \frac{1}{2}\varepsilon.
  $$
  There exists $s \geq t$ such that $\nu(\{s\}) =0$, and from (v) we get
  $$
  \lim_{n\to\infty}\nu_n((s,\infty))=\nu((s, \infty)) \geq 0
  \quad\text{and}\quad
  \lim_{n\to\infty}\int_s^\infty\!\!\!\!\log(u)\,\nu_n(du)
  =\int_s^\infty\!\!\!\!\log(u)\,\nu(du) \leq
  \frac{1}{2}\varepsilon.
  $$
  If $\nu((s,\infty))>0$ then \eqref{eq:defkn} holds with $\sigma:=\nu((s,
  \infty))$ and $k_n := \FLOOR{n\nu_n((s, \infty))}$. Otherwise,
  $\nu((s, \infty))=0$, and if $k'_n :=\FLOOR{n\nu_n ((s,\infty))}$
  then $\lim_{n\to\infty}k'_n/n=0$, while for any $ \delta > 0$,
  $$
  \frac{1}{n}\sum_{i=1}^{k'_n+\FLOOR{n\delta}}\log(b_{n,i})
  \leq \frac{\varepsilon}{2}+\delta\log(s).
  $$
  Taking $k_n := k'_n + \FLOOR{n\delta}$ with $\delta$ small enough, we deduce
  that \eqref{eq:defkn} holds. We have thus shown that \eqref{eq:defkn} holds
  in all case. Now, from (ii) and \eqref{eq:defkn} we get for every $1 \leq k
  \leq k_n$,
  $$
  \frac{1}{n} \sum_{i=1}^{k} \log(a_{n,i}) < \varepsilon.
  $$
  In particular, by using (i), we get $\log(a_{n,k_n})\leq\varepsilon n /
  k_n$ and 
  $$
  \int_{e^{ \varepsilon  n /  k_n }}^\infty\!\log(u)\mu_n(du) < \varepsilon.
  $$
  Since $\varliminf_{n\to\infty} k_n / n = \delta >0$, we deduce that
  \eqref{eq:unifconvmu} holds with $t := e^{ \varepsilon \delta }$. Now, by
  following the same reasoning, with (ii) replaced by (iii), we obtain that
  for all $\varepsilon >0$, there exists $0 < t < 1$ such that for $n\gg1$,
  $$
  -\int_0^t\!\log(s)\mu_n(ds) < \varepsilon,
  $$
  which is the counterpart of \eqref{eq:unifconvmu} needed for (jj).
\end{proof}

\begin{rem}[Other fundamental aspects of the logarithmic potential]
  The logarithmic potential is related to the Cauchy-Stieltjes transform of
  $\mu$ via
  $$
  S_\mu(z) %
  :=\int_{\mathds{C}}\!\frac{1}{z'-z}\,\mu(dz') %
  =(\partial_x-i\partial_y)U_\mu(z)
  \quad\text{and thus}\quad
  (\partial_x+i\partial_y)S_\mu=-2\pi\mu
  $$
  in $\mathcal{D}'(\mathds{C})$. The term ``logarithmic potential'' comes from
  the fact that $U_\mu$ is the electrostatic potential of $\mu$ viewed as a
  distribution of charges in $\mathds{C}\equiv\mathds{R}^2$ \cite{MR1485778}.
  The logarithmic energy
  $$
  \mathcal{E}(\mu)
  :=\int_{\mathds{C}}\!U_\mu(z)\,\mu(dz)
  =-\int_{\mathds{C}}\int_{\mathds{C}}\!\log\ABS{z-z'}\,\mu(dz)\mu(dz')
  $$
  is up to a sign the Voiculescu free entropy of $\mu$ in free probability
  theory \cite{MR1887698}. The circular law $\mathcal{U}_\sigma$ minimizes
  $\mu\mapsto\mathcal{E}(\mu)$ under a second moment constraint
  \cite{MR1485778}. In the spirit of \eqref{eq:UESD} and beyond matrices, the
  Brown \cite{MR866489} spectral measure of a nonnormal bounded operator $a$ is
  $\mu_a:=(-4\pi)^{-1}\Delta\int_0^\infty\!\log(t)\,\nu_{a-zI}(dt)$ where
  $\nu_{a-zI}$ is the spectral distribution of the self-adjoint operator
  $(a-zI)(a-zI)^*$. Due to the logarithm, the Brown spectral measure $\mu_a$
  depends discontinuously on the $*$-moments of $a$ \cite{MR1876844,MR1929504}.
  For random matrices, this problem is circumvented in the Girko Hermitization
  by requiring a uniform integrability, which turns out to be a.s.\ satisfied
  for random matrices such as $n^{-1/2}X$ or $\sqrt{n}M$.
\end{rem}  

\section{General spectral estimates}\label{se:spest}

We gather in this section useful lemmas on deterministic matrices. We provide
mainly references for the most classical results, and sometimes proofs for the
less classical ones.

\begin{lem}[Basic inequalities \cite{MR1288752}]\label{le:basic}
If $A$ and $B$ are $n\times n$ complex matrices then
\begin{equation}\label{eq:basic0}
  s_1(AB)\leq s_1(A)s_1(B)
  \quad\text{and}\quad
  s_1(A+B)\leq s_1(A)+s_1(B)
\end{equation}
and
\begin{equation}\label{eq:basic1}
  \max_{1\leq i\leq n}|s_i(A)-s_i(B)|\leq s_1(A-B)
\end{equation}
and
\begin{equation}\label{eq:basic2}
  s_n(AB)\geq s_n(A)s_n(B).
\end{equation}
Moreover, if $A=D$ is diagonal, then for every $1\leq i\leq n$
\begin{equation}\label{eq:basic3}
  s_n(D)s_i(B) \leq s_i(DB) \leq s_1(D)s_i(B).
\end{equation}
\end{lem}

\begin{lem}[Rudelson-Vershynin row bound]\label{le:rvdist}
  Let $A$ be a complex $n\times n$ matrix with rows $R_1,\ldots,R_n$. Define
  the vector space $R_{-i}:=\mathrm{span}\{R_j:j\neq i\}$. We have then
  $$
  n^{-1/2}\min_{1\leq i\leq n}\mathrm{dist}(R_i,R_{-i}) \leq
  s_n(A) \leq \min_{1\leq i\leq n}\mathrm{dist}(R_i,R_{-i}).
  $$
\end{lem}

The argument behind lemma \ref{le:rvdist} is buried in \cite{MR2407948}. We
give a proof below for convenience.

\begin{proof}[Proof of lemma \ref{le:rvdist}]
  Since $A,A^\top$ have same singular values, one can consider the columns
  $C_1,\ldots,C_n$ of $A$ instead of the rows. For every column vector
  $x\in\mathds{C}^n$ and $1\leq i\leq n$, the triangle inequality and the
  identity $Ax=x_1C_1+\cdots+x_n C_n$ give
  $$
  \Vert Ax\Vert_2
  \geq \mathrm{dist}(Ax,C_{-i}) 
  =\min_{y\in C_{-i}} \left\Vert Ax-y\right\Vert_2 
  =\min_{y\in C_{-i}} \left\Vert x_iC_i-y\right\Vert_2 
  =\vert x_i\vert\mathrm{dist}(C_i,C_{-i}).
  $$
  If $\left\Vert x\right\Vert_2 =1$ then necessarily $\vert x_i\vert \geq
  n^{-1/2}$ for some $1\leq i\leq n$ and therefore
  $$
  s_n(A) %
  =\min_{\left\Vert x\right\Vert_2 =1}\left\Vert Ax\right\Vert_2 %
  \geq n^{-1/2}\min_{1\leq i\leq n}\mathrm{dist}(C_i,C_{-i}).
  $$
  Conversely, for every $1\leq i\leq n$, there exists a vector $y$ with
  $y_i=1$ such that
  $$
  \mathrm{dist}(C_i,C_{-i}) %
  =\left\Vert y_1C_1+\cdots+y_nC_n\right\Vert_2 %
  =\left\Vert Ay\right\Vert_2 %
  \geq \left\Vert y\right\Vert_2 %
  \min_{\left\Vert x\right\Vert_2=1}\left\Vert Ax\right\Vert_2 %
  \geq s_n(A)
  $$
  where we used the fact that $\Vert y\Vert^2_2 = |y_1|^2+\cdots+|y_n|^2\geq
  |y_i|^2=1$.
\end{proof}

Recall that the singular values $s_1(A),\ldots,s_{n'}(A)$ of a rectangular
$n'\times n$ complex matrix $A$ with $n'\leq n$ are defined by
$s_i(A):=\lambda_i(\sqrt{AA^*})$ for every $1\leq i\leq n'$.

\begin{lem}[Tao-Vu negative second moment {\cite[lem.
    A4]{tao-vu-cirlaw-bis}}]\label{le:tvneg}
  If $A$ is a full rank $n'\times n$ complex matrix ($n'\leq n$) with rows
  $R_1,\ldots,R_{n'}$, and $R_{-i}:=\mathrm{span}\{R_j:j\neq i\}$, then
   $$ 
   \sum_{i=1}^{n'}s_i(A)^{-2}=\sum_{i=1}^{n'}\mathrm{dist}(R_i,R_{-i})^{-2}.
   $$  
\end{lem}

\begin{lem}[Cauchy interlacing by rows deletion \cite{MR1288752}]
  \label{le:cauchy}
  Let $A$ be an $n\times n$ complex matrix. If $B$ is $n'\times n$, obtained
  from $A$ by deleting $n-n'$ rows, then for every $1\leq i\leq n'$,
  $$
  s_i(A)\geq s_i(B)\geq s_{i+n-n'}(A).
  $$
\end{lem}

Lemma \ref{le:cauchy} gives $[s_{n'}(B),s_1(B)]\subset[s_n(A),s_1(A)]$, i.e.\
row deletions produce a compression of the singular values interval. Another
way to express this phenomenon consists in saying that if we add a row to $B$
then the largest singular value increases while the smallest is diminished.
Closely related, the following result on finite rank additive perturbations.
If $A$ is an $n\times n$ complex matrix, let us set $s_i(A):=+\infty$ if $i<1$
and $s_i(A):=0$ if $i>n$.

\begin{lem}[Thompson-Lidskii interlacing for finite
  rank perturbations \cite{MR0407051}]\label{le:thompson}
  For any $n\times n$ complex matrices $A$ and $B$ with
  $\mathrm{rank}(A-B)\leq k$, we have, for any $i\in\{1,\ldots,n\}$,
  \begin{equation}\label{eq:thompson}
  s_{i-k}(A) \geq s_i(B) \geq s_{i+k}(A).
  \end{equation}
\end{lem}

Even if lemma \ref{le:thompson} gives nothing on the extremal singular values
$s_i(B)$ where $i\leq k$ or $n-i< k$, it provides however the useful ``bulk''
inequality $\NRM{F_A-F_B}_\infty\leq \mathrm{rank}(A-B)/n$ where $F_A$ and
$F_B$ are the cumulative distribution functions of $\nu_A$ and $\nu_B$
respectively.

\begin{lem}[Weyl inequalities \cite{MR0030693}]\label{le:weyl}
  For every $n\times n$ complex matrix $A$, we have 
  \begin{equation}\label{eq:weyl0}
    \prod_{i=1}^k|\lambda_i(A)|\leq \prod_{i=1}^ks_i(A)
    \quad\text{and}\quad
    \prod_{i=k}^ns_i(A) \leq \prod_{i=k}^n|\lambda_i(A)| 
  \end{equation}
  for all $1\leq k\leq n$, with equality for $k=n$. In particular, by
  viewing $\ABS{\det(A)}$ as a volume,
  \begin{equation}\label{eq:weyl1}
    |\det(A)|=\prod_{k=1}^n|\lambda_k(A)|=\prod_{k=1}^ns_k(A)
    =\prod_{k=1}^n\mathrm{dist}(R_k,\mathrm{span}\{R_1,\ldots,R_{k-1}\})
  \end{equation}
  where $R_1,\ldots,R_n$ are the rows of $A$. Moreover, for every increasing
  function $\varphi$ from $(0,\infty)$ to $(0,\infty)$ such that
  $t\mapsto\varphi(e^t)$ is convex on $(0,\infty)$ and
  $\varphi(0):=\lim_{t\to0^+}\varphi(t)=0$, we have
  \begin{equation}\label{eq:weyl2}
    \sum_{i=1}^k\varphi(|\lambda_i(A)|^2) \leq \sum_{i=1}^k\varphi(s_i(A)^2)
  \end{equation}
  for every $1\leq k\leq n$. In particular, with $\varphi(t)=t$ for every
  $t>0$ and $k=n$, we obtain
  \begin{equation}\label{eq:weyl3}
    \sum_{k=1}^n|\lambda_k(A)|^2 \leq \sum_{k=1}^ns_k(A)^2
    =\mathrm{Tr}(AA^*)=\sum_{i,j=1}^n|A_{i,j}|^2.
  \end{equation}
\end{lem}

It is worthwhile to mention that \eqref{eq:thompson} and \eqref{eq:weyl0} are
optimal in the sense that every sequences of numbers satisfying these
inequalities are associated to matrices, see \cite{MR0407051,MR0061573}.

\section{Additional lemmas}\label{se:adle}

Lemma \ref{le:perturb} below is used in the proof of theorem
\ref{th:singular}. We omit its proof since it follows for instance quite
easily from the Paul L\'evy criterion on characteristic functions.

\begin{lem}[Convergence under uniform perturbation]\label{le:perturb}
  Let $(a_{n,k})_{1\leq k\leq n}$ and $(b_{n,k})_{1\leq k\leq n}$ be
  triangular arrays of complex numbers. Let $\mu$ be a probability measure on
  $\mathds{C}$.
  $$
  \text{If}\quad %
  \frac{1}{n}\sum_{k=1}^n%
  \delta_{a_{n,k}}\underset{n\to\infty}{\overset{\mathscr{C}_b}{\longrightarrow}}\mu %
  \quad\text{and}\quad %
  \lim_{n\to\infty}\max_{1\leq k\leq n}|a_{n,k}-b_{n,k}|=0 %
  \quad\text{then}\quad %
  \frac{1}{n}\sum_{k=1}^n%
  \delta_{b_{n,k}}\underset{n\to\infty}{\overset{\mathscr{C}_b}{\longrightarrow}}\mu.
  $$
\end{lem}

Lemma \ref{le:concdist} below is used for the rows of random matrices in the
proof of theorem \ref{th:circular}.

\begin{lem}[Tao-Vu distance lemma {\cite[prop.
    5.1]{tao-vu-cirlaw-bis}}]\label{le:concdist} Let $(X_i)_{i\geq1}$ be
  i.i.d.\ random variables on $\mathds{C}$ with finite positive variance
  $\sigma^2:=\mathds{E}(|X_1-\mathds{E}X_1|^2)$. For $n\gg1$ and every
  deterministic subspace $H$ of $\mathds{C}^n$ with $1 \leq \mathrm{dim}(H)
  \leq n - n ^{0.99}$, setting $R:=(X_1,\ldots,X_n)$,
    $$
    \mathds{P}\PAR{\mathrm{dist}(R,H) \leq \frac{\sigma}{2}\sqrt {n-\mathrm{dim}(H)}} %
    \leq \exp(-n^{0.01}).
    $$
\end{lem}

The proof of lemma \ref{le:concdist} is based on a concentration inequality
for convex Lipschitz functions and product measures due to Talagrand
\cite{MR1361756}, see also \cite[cor. 4.9]{MR1849347}. The power $0.01$ is
used here to fix ideas and is obviously not optimal. This is more than enough
for our purposes (proof of theorem \ref{th:circular}). A careful reading of
the proof of theorem \ref{th:circular} shows that a polynomial bound on the
probability with a large enough power on $n$ suffices.

We end up this section by a lemma used in the proof of theorem
\ref{th:least-singular}. 

\begin{lem}[A special matrix]\label{le:A}
  For every $w\in\mathds{C}$, let us define
  the $n\times n$ complex matrix 
  $$
  A_w=I-w
  \begin{pmatrix}
    1 & 0 & \cdots & 0\\
    \vdots & \vdots &  & \vdots\\
    1 & 0 & \cdots & 0
  \end{pmatrix}.
  $$
  Then for every $z\in\mathds{C}$ we have
  $s_2(A_{\frac{z}{\sqrt{n}}})=\cdots=s_{n-1}(A_{\frac{z}{\sqrt{n}}})=1$ for
  $n\gg1$ while
  $$
  \lim_{n\to\infty}s_n(A_{\frac{z}{\sqrt{n}}}) %
  = \lim_{n\to\infty}s_1(A_{\frac{z}{\sqrt{n}}})^{-1} %
  = \frac{\sqrt{2}}{\sqrt{2+|z|^2+|z|\sqrt{4+|z|^2}}}
  $$
  and the convergence is uniform on every compact subset of $\mathds{C}$.
\end{lem}

\begin{proof}
  Note that $A_0=I$ and $A_wA_{w'}=A_{ww'-(w+w')}$ for every
  $w,w'\in\mathds{C}$. Moreover, $A_w$ is invertible if and only if $w\neq1$
  and in that case $(A_w)^{-1}=A_{w/(w-1)}$. It is a special case of the
  Sherman-Morrison formula for the inverse of rank one perturbations. It is
  immediate to check that $s_1(A_w-I)=\Vert A_w-I\Vert_{2\to2}=\sqrt{n}|w|$
  for every $w\in\mathds{C}$. An elementary explicit computation reveals that
  the symmetric matrix $A_wA_w^*-I$ has rank at most $2$, and thus $A_w$ has
  at least $n-2$ singular values equal to $1$ and in particular $s_n(A_w) \leq
  1 \leq s_1(A_w)$. From now, let us fix $z\in\mathds{C}$ and set
  $w=n^{-1/2}z$ and $A=A_w$ for convenience. The matrix $A$ is nonsingular for
  $n\gg1$ since $w\to0$ as $n\to\infty$. Also, we have $s_n(A)>0$ for $n\gg1$.
  Since $A$ is lower triangular with eigenvalues $1-w,1,\ldots,1$, by
  \eqref{eq:weyl1},
  \begin{equation}\label{eq:uv}
    |1-w| %
    = \prod_{i=1}^n|\lambda_i(A)| %
    = |\det(A)| %
    = \prod_{i=1}^ns_i(A) %
    = u_-u_+
  \end{equation}
  where $u_-\leq u_+$ are two singular values of $A$. We have also
  $$
  u_-^2+u_+^2+(n-2)
  = s_1(A)^2+\cdots+s_n(A)^2 %
  = \mathrm{Tr}(AA^*) %
  = |1-w|^2+(n-1)(1+|w|^2)
  $$
  which gives $u_-^2+u_+^2=1+|1-w|^2+(n-1)|w|^2$. Combined with \eqref{eq:uv},
  we get that $u_\pm^2$ are the solution of $X^2-(1+(n-1)|w|^2+|1-w|^2)X +
  |1-w|^2=0$. This gives
  $$
  2u_\pm^2=2+|z|^2+O(n^{-1/2})\pm|z|\sqrt{4+|z|^2+O(n^{-1/2})}
  $$
  and the $O(n^{-1/2})$ is uniform in $z$ on every compact. From this formula
  we get that $u_-\leq 1$ and $u_+\geq1$ for $n\gg1$, and thus
  $u_-=s_n(A)$ and $u_+=s_1(A)$.
\end{proof}

The result of lemma \ref{le:A} is more than enough for our purposes. More
precisely, a careful reading of the proof of theorem \ref{th:least-singular}
shows that a polynomial (in $n$) lower bound on $s_n(A_{n^{-1/2}z})$ for
$n\gg1$, uniformly on compact sets on $z$, is actually enough.

\section*{Acknowledgments} 

DC is grateful to O. Gu\'edon, M. Krishnapur, and T. Tao for helpful
discussions. CB and DC are grateful to Dipartimento di Matematica Universit\`a
Roma Tre for kind hospitality.

\providecommand{\bysame}{\leavevmode\hbox to3em{\hrulefill}\thinspace}
\providecommand{\MR}{\relax\ifhmode\unskip\space\fi MR }
\providecommand{\MRhref}[2]{%
  \href{http://www.ams.org/mathscinet-getitem?mr=#1}{#2}
}
\providecommand{\href}[2]{#2}

\begin{center}
  \begin{figure}[htbp]
    \includegraphics[scale=.36]{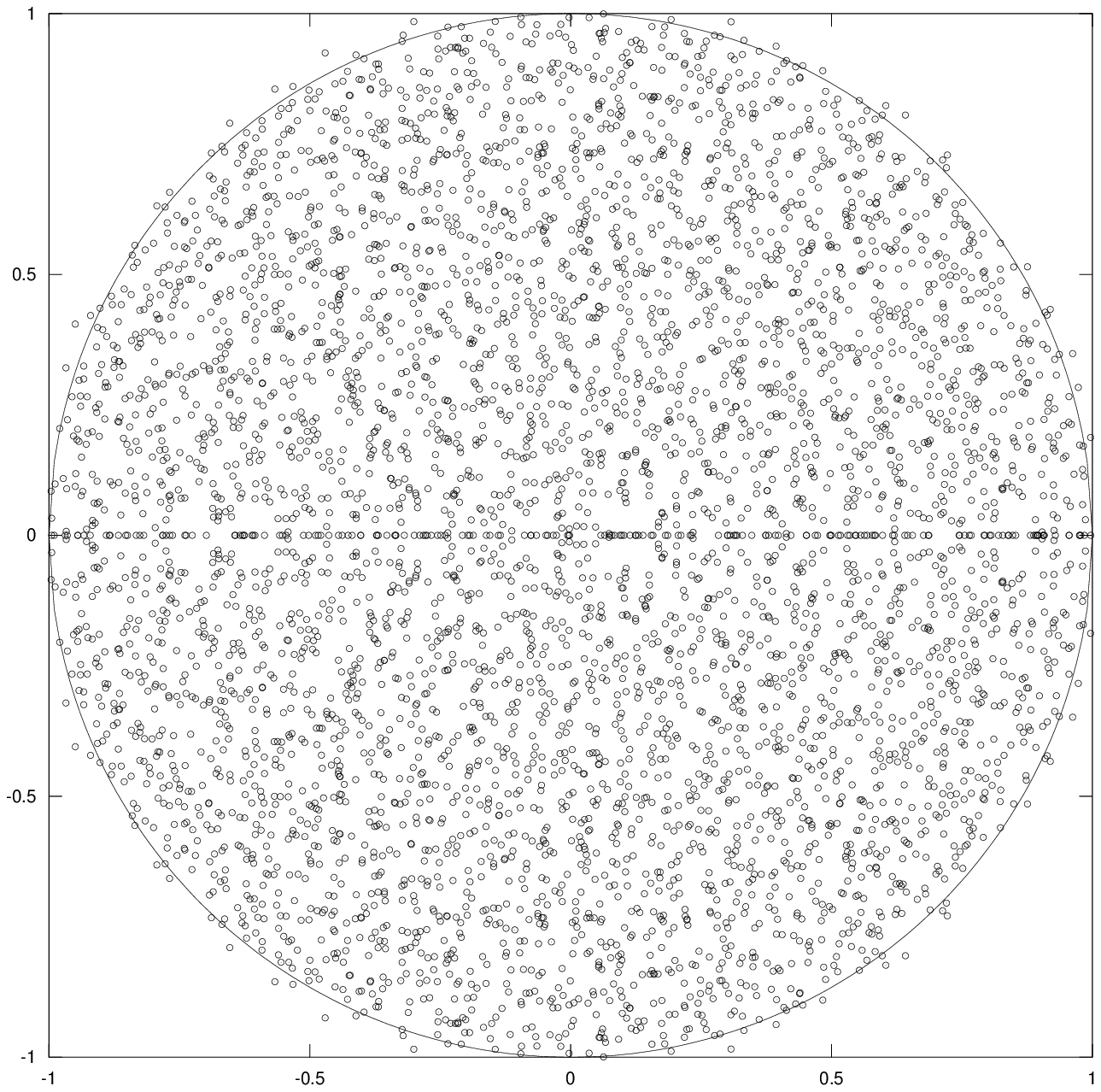}
    \includegraphics[scale=.36]{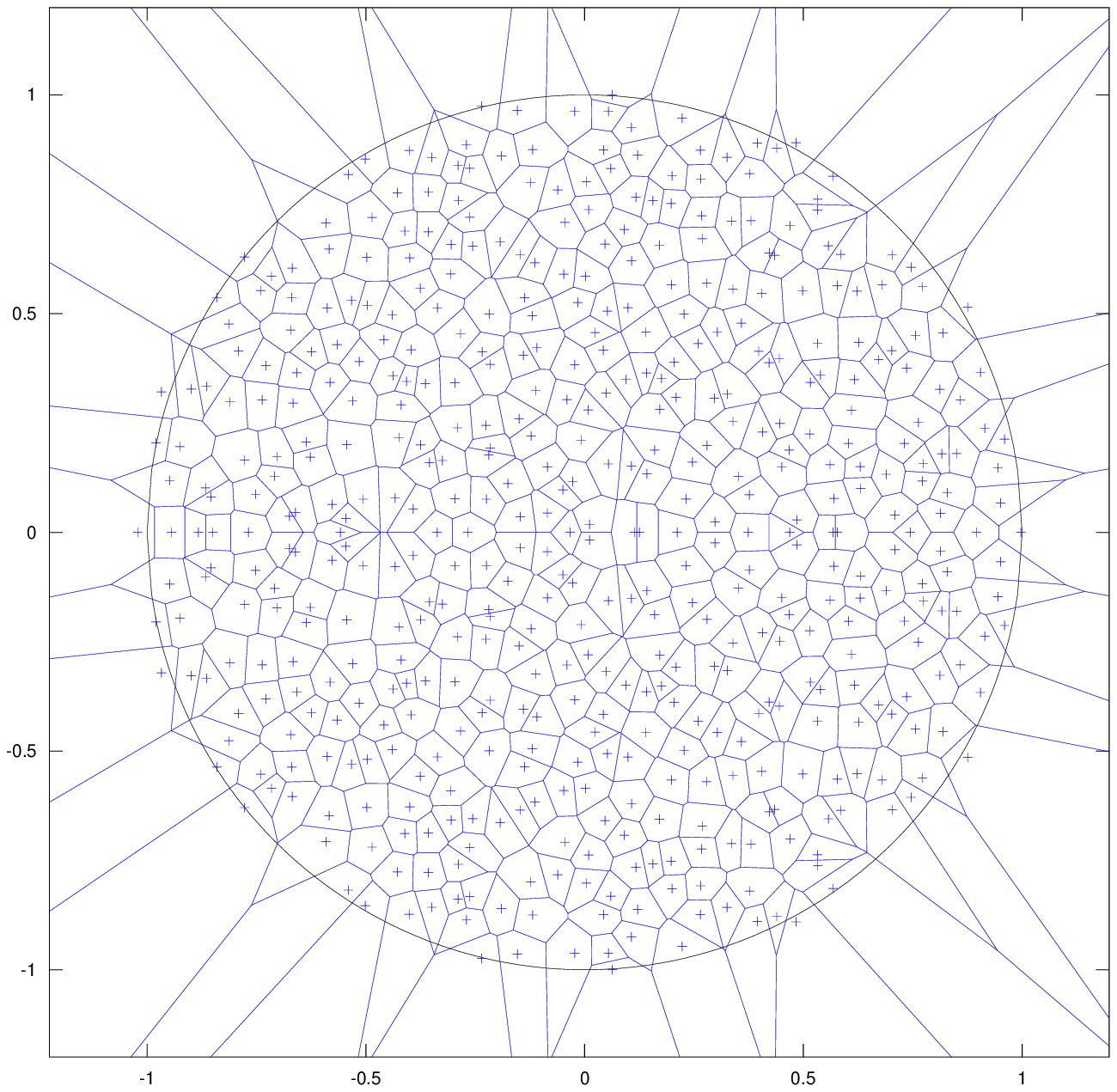}
    \caption{Here we have fixed $n=250$, and $X_{1,1}$ follows the Bernoulli
      law $\frac{1}{2}(\delta_0+\delta_1)$. In both graphics, the solid circle
      has radius $m^{-1}\sigma=1$. The left hand side graphic is the
      superposition of the plot of $\lambda_2,\ldots,\lambda_n$ for $10$
      i.i.d.\ simulations of $\sqrt{n}M$, made with the GNU Octave free
      software. The right hand side graphic is the Vorono{\"{\i}} tessellation
      of $\lambda_2,\ldots,\lambda_n$ for a single simulation of $\sqrt{n}M$.
      Since $\sqrt{n}M$ has real entries, its spectrum is symmetric with
      respect to the real axis. On the left hand side graphic, it seems that
      the spectrum is slightly more concentrated on the real axis. This
      phenomenon, which disappears as $n\to\infty$, was already described for
      random matrices with i.i.d.\ real Gaussian entries by Edelman
      \cite{MR1437734}, see also the work of Akemann and Kanzieper
      \cite{MR2363393}. Our simulations suggest that theorem \ref{th:circular}
      remains valid beyond the bounded density assumption.}
    \label{fi:simus}
  \end{figure} 
\end{center}

\end{document}